\documentclass[11pt,twoside]{amsart}

\usepackage{latexsym}
\usepackage{amsmath}
\usepackage{amssymb}
\usepackage{amsthm}
\usepackage{amscd}
\usepackage[all]{xy}
\usepackage{enumerate} 
\usepackage{amssymb} 
\usepackage{mathrsfs}

\def\N{{\mathscr{N}}}
\def\NN{{\mathbb{N}}}
\def\ZZ{{\mathbb{Z}}}
\def\QQ{{\mathbb{Q}}}
\def\C{{\mathscr{C}}}

\def\K{{\mathscr{K}}}

\def\PP{{\mathbb{P}}}
\def\E{{\mathscr{E}}}
\def\O{{\mathscr{O}}}
\def\U{{\mathscr{U}}}
\def\Q{{\mathscr{Q}}}

\def\L{{\mathscr{L}}}
\def\S{{\mathscr{S}}}

\usepackage{latexsym}

\newtheorem{them}{Theorem}[section]

\newtheorem{pro}[them]{Proposition}

\newtheorem{lem}[them]{Lemma}

\newtheorem{rem}[them]{Remark}

\newtheorem{defi}[them]{Definition}

\newtheorem{conj}[them]{Conjecture}

\newtheorem{nota}[them]{Notation}

\newtheorem{Ass}[them]{Assumptions}

\begin{document}

\title[$\PP^1$-bundles over projective manifold of Picard number one]{$\PP^1$-bundles over projective manifolds of Picard number one each of which admit another smooth morphism of relative dimension one}
\author{Kiwamu Watanabe}
\date{February 7, 2012}

\address{Course of Mathematics, Programs in Mathematics, Electronics and Informatics, 
Graduate School of Science and Engineering, Saitama University.
Shimo-Okubo 255, Sakura-ku Saitama-shi, 338-8570 JAPAN}
\email{kwatanab@rimath.saitama-u.ac.jp}


\maketitle



\begin{abstract}
We give the complete classification of $\PP^1$-bundles over projective manifolds of Picard number one each of which admit another smooth morphism of relative dimension one. 

\end{abstract}

\tableofcontents

\section{Introduction}
R. Mu$\rm \tilde{n}$oz, G. Occhetta and L. Sol$\rm \acute{a}$ Conde studied rank $2$ vector bundles on Fano manifolds in \cite{MOS}. In their paper \cite[Theorem~6.5]{MOS}, they obtained a complete list of $\PP^1$-bundles over Fano manifolds with $b_2=b_4=1$ that have another second $\PP^1$-bundle structure. 
The purpose of this paper is to generalize their result. Actually, we give the complete classification of $\PP^1$-bundles over projective manifolds of Picard number $1$ 
each of which admit another smooth morphism of relative dimension $1$.
Our main result is the following:

\begin{them}\label{MT} 

Let $X$ be a complex projective manifold of Picard number $\rho=1$ and $\E$ a rank $2$ vector bundle on $X$. Assume that $Z:=\PP(\E) \rightarrow X$ admits another smooth morphism $Z \rightarrow Y$ of relative dimension $1$ and $n:=\dim X \geq 2$. Then,
\begin{enumerate}
\item[{\rm (I)}] $X$ and $Y$ are Fano manifolds of $\rho=1$ and there exists a rank $2$ vector bundle $\E'$ on $Y$ such that $Z \rightarrow Y$ is given by $\PP_Y(\E')$.   
\end{enumerate}
Furthermore,  
\begin{enumerate}
\item[{\rm (II)}]  if $\E$ and $\E'$ are normalized by twisting with line bundles (i.e., $c_1=0$ or $-1$), then $((X, \E), (Y, \E'))$ is one of the following, up to changing the pairs $(X, \E)$ and $(Y, \E')$:
\begin{enumerate}
\item $((\PP^2, T_{\PP^2}), (\PP^2, T_{\PP^2}))$, where $T_{\PP^2}$ is the tangent bundle of the projective plane $\PP^2$,
\item $((\PP^3, \N),(Q^3, \S))$, where $\N$ is the null-correlation bundle on $\PP^3$ (see \cite{OSS}) and $\S$ is the restriction to the $3$-dimensional quadric $Q^3$ of the universal quotient bundle of the Grassmannian $G(1,\PP^3)$,
\item $((Q^5, \C), (K(G_2), \Q))$, where $\C$ is a Cayley bundle on $Q^5$ (see \cite{Ota}), $K(G_2)$ is the $5$-dimensional Fano homogeneous contact manifold of type $G_2$ which is a linear section of the Grassmannian $G(1,\PP^6)$ and $\Q$ the restriction of the universal quotient bundle on $G(1,\PP^6)$. 
\end{enumerate}
\end{enumerate}

Consequently, $Z$ is the full-flag manifold of type $A_2$, $B_2$ or $G_2$.  
In particular, $X$, $Y$ and $Z$ are rational homogeneous manifolds. 
\end{them} 


Another motivation of our main result is the following conjecture proposed by Campana and Peternell: 
\begin{conj}[{\cite{CP}}] A Fano manifold $M$ with nef tangent bundle is homogeneous.
\end{conj}
This conjecture is true in dimension $\leq 4$. The most difficulty lies in the case where $M$ is a Fano $4$-fold of $\rho=1$ which carries a rational curve $C$ with $-K_M.C=3$. In this case, N. Mok \cite{Mok} proved the conjecture under the additional assumption that $b_4(M)=1$.    
 Later on, J. M. Hwang pointed out that the assumption $b_4(M)=1$ can be removed in \cite{Hwang}. More generally, they obtained the following result. We will also prove this as a corollary of Theorem~\ref{MT}.
\begin{them}[{\cite[Main Theorem]{Mok}, \cite{Hwang}}]\label{HM} Let $M$ be a Fano manifold of $\rho=1$ with nef tangent bundle. Assume that $M$ carries a rational curve $C$ such that $-K_M.C=3$. Then $M$ is isomorphic to $\PP^2$, $Q^3$ or $K(G_2)$.
\end{them}


The contents of this paper are organized as follows. 
Section~2 is devoted to study the structures of the Chow group $A_2(X)_{\QQ}$ of $2$-dimensional cycles with $\QQ$-coefficients and its quotient $N^2(X)_{\QQ}$ by numerical equivalence, according to a similar argument as in \cite{Hwang}. In Section~3, we give two computations of the discriminant $\Delta(\E):=c_1^2(\E)-4c_2(\E)$ which are based on ideas in \cite{Mok} and \cite{MOS} (see Proposition~\ref{key} and \ref{delta}). In Section~4, by comparing the computational results, we narrow down the possible values of some invariants of $X$ and $Y$. Then we can show the existence of a rank $2$ vector bundle $\E'$ on $Y$ such that $Z=\PP_Y(\E')$ (Proposition~\ref{n35}). The main novelty of this paper is to show the existence of $\E'$ by using the above two computational results of the discriminant. After this paper was submitted, the draft of \cite{MOS} was revised and they obtained the same result (see \cite[Lemma~6.2]{MOS}). Since $Z$ admits double $\PP^1$-bundle structures $\pi: Z \rightarrow X$ and $\phi: Z \rightarrow Y$, the same argument as in \cite{MOS} implies Theorem~\ref{MT} (Theorem~\ref{conc}). In the final section, we will show Theorem~\ref{HM} as a corollary of Theorem~\ref{MT}.

In this paper, we use notation as in \cite{Ha} and every point on a variety we deal with is a closed point. We work over the field of complex numbers. \\  

{\bf Acknowledgements.} 
This work was done while the author was visiting the University of Freiburg. The author would like to express his gratitude to his host Prof. Stefan Kebekus for all his assistance and his kindness. He would also like to thank the staff in the department of mathematics at the University of Freiburg, especially Patrick Graf for his support and encouragement. My deepest appreciation goes to Luis E. Sol$\rm \acute{a}$ Conde and Roberto Mu$\rm \tilde{n}$oz for inviting me to visit Universidad Rey Juan Carlos. The author would like to express his gratitude to Gianluca Occhetta for sending the final version of their paper \cite{MOS}.
 The author is partially supported by Research Fellowships of the Japan Society for the Promotion of Science for Young Scientists.

\section{A Part of Theorem~\ref{MT}~(I) and Structures of $A_2(X)_{\QQ}$ and $N^2(X)_{\QQ}$}

First, we partially prove Theorem~\ref{MT}~(I). We will complete the proof of Theorem~\ref{MT}~(I) in Proposition~\ref{n35}. 
\begin{pro}\label{MTI} 
Let $X$ be a projective manifold of $\rho_X=1$ and $\E$ a rank $2$ vector bundle on $X$. Assume that $\pi: Z=\PP(\E) \rightarrow X$ admits another smooth morphism $\phi: Z \rightarrow Y$ of relative dimension $1$ and $n:=\dim X \geq 2$. Then 
\begin{enumerate}
\item $Y$ is a Fano manifold of $\rho_Y=1$, 
\item $\phi^{-1}(y)$ is isomorphic to $\PP^1$ for every $y \in Y$, and
\item $X$ is also a Fano manifold of $\rho_X=1$.
\end{enumerate}   
\end{pro}

It is easy to see $\rm (i)$. In fact, since $\rho_Z=2$, it turns out that $\rho_Y=1$. Additionally, $Y$ is covered by rational curves which are the images of fibers of $\pi$. Thus $Y$ is a Fano manifold of $\rho_Y=1$.

We use the following lemma to prove Proposition~\ref{MTI} $\rm (ii)$: 

\begin{lem}\label{V} Let $X$ be a projective manifold and $Y$ a Fano manifold of $\rho=1$. Assume that a projective manifold $Z$ admits two different smooth morphisms $\pi: Z \rightarrow X$ and $\phi: Z \rightarrow Y$ of relative dimension $1$ and $\pi^{-1}(x)$ is isomorphic to $\PP^1$ for every $x \in X$.  Given $y \in Y$, define inductively 
\begin{enumerate}
\item $V_y^0:=\{y\}$, and 
\item $V_y^{m+1}:={\phi}({\pi}^{-1}({\pi}({\phi}^{-1}(V_y^m))))$.
\end{enumerate}  
Then there exists a natural number $l$ such that $V_y^l=Y$.
\end{lem}

\begin{proof} The idea of this proof is in \cite{KMM1}. Since $V_y^k$ is an irreducible closed subset of $Y$, it is sufficient to show that $V_y^k=Y$ provided $\dim V_y^k=\dim V_y^{k+1}$. Remark that $\dim V_y^k$ is independent of the choice of $y \in Y$. It follows from flatness of $\pi$ and $\phi$.
Assume that $\dim V_y^k=\dim V_y^{k+1}$ for any $y \in Y$. Then $V_y^k=V_y^{k+1}$. 
Supposing that $V_y^k$ does not coincide with $Y$, we shall derive a contradiction.  
Let $q$ be the codimension of $V_y^k$ in $Y$ and $T \subset Y$ a $(q-1)$-dimensional projective subvariety. From our assumption, we have $q \geq 1$. Denote $\bigcup_{y \in T}V_y^k$ by $A$. Since $\rho_Y=1$, $A$ is an ample divisor on $Y$. Hence, for any point $x \in X$, $\phi(\pi^{-1}(x)) \cap A \neq \emptyset$, then there exists a point $y_x \in T$ such that $\phi(\pi^{-1}(x)) \cap V_{y_x}^k \neq \emptyset$. This implies that $\phi(\pi^{-1}(x))$ is contained in $V_{y_x}^{k+1}=V_{y_x}^k \subset A$.  However this contradicts the surjectivity of $\phi$. 
\end{proof}

\begin{proof}[Proof of Proposition~\ref{MTI}(ii) and (iii)]
Assume that there exists an irrational fiber of $\phi$. Then every fiber of $\phi$ is not rational. 
Let $f$ be a fiber of $\pi$ and $\nu$ the restriction of $\phi$ to $f \cong \PP^1$. Consider a smooth family of curves $\nu^{\ast}Z \rightarrow f\cong \PP^1$. Since a fiber of $\nu^{\ast}Z \rightarrow f\cong \PP^1$ is not rational, the family is isotrivial. Furthermore, the family is trivial by virtue of the simply-connectedness of $\PP^1$. It turns out that $\pi(\phi^{-1}(y_1))=\pi(\phi^{-1}(y_2))$ for any $y_1, y_2 \in \phi(\pi^{-1}(x))$ provided we fix a point $x \in X$. From Lemma~\ref{V}, it follows that any two point can be connected by a chain of rational curves $\phi(\pi^{-1}(x))$ of finite length. Hence we see that $\pi(\phi^{-1}(y_1))=\pi(\phi^{-1}(y_2))$ for any $y_1, y_2 \in Y$. However, this is a contradiction to the surjectivity of $\pi$ and $\dim \geq 2$. As a consequence, every fiber of $\phi$ is rational. 
Now $\rm (iii)$ follows in a similar way to $\rm (i)$. 
\end{proof}


According to a similar argument as in \cite[Sect. 4]{Hwang}, we prove the following:
\begin{pro}\label{A2} 
Let $X$ be an $n$-dimensional Fano manifold of $\rho=1$ and $\E$ a rank $2$ vector bundle on $X$. Assume that $Z:=\PP(\E) \rightarrow X$ admits another smooth morphism $\phi: Z \rightarrow Y$ whose fiber is isomorphic to $\PP^1$. If $n \geq 2$, then $A_2(X)_{\QQ}$ and $N^2(X)_{\QQ}$ are isomorphic to a $1$-dimensional vector space $\QQ$ over the field of rational numbers.
\end{pro}

\begin{proof} For a point $y \in Y$, we define inductively the varieties $W_y^{k}$ and $\widetilde{W^{k}_y}$ as follows:
\begin{enumerate}
\item $W_y^0:=\phi^{-1}(y)$, $\widetilde{W^0_y}:=W_y^0 \times_XZ$ and 
\item $W_y^{k}:=\widetilde{W^{k-1}_y} \times_YZ$, $\widetilde{W^{k}_y}:=W_y^{k} \times_XZ$.
\end{enumerate}
Remark that $\widetilde{W^{k-1}_y}$ has a natural morphism to $Y$ defined by the composition of a projection $\widetilde{W^{k-1}_y} \rightarrow Z$ and $\phi: Z \rightarrow Y$. On the other hand, $W_y^{k}$ admits a natural morphism to $X$ by the composition of a projection $W_y^{k} \rightarrow Z$ and $\pi: Z \rightarrow X$. Hence we can define $W_y^{k}$ and $\widetilde{W^{k}_y}$ as above. 

For a point $y \in Y$, the image of the composition of a projection $\widetilde{W^{k}_y} \rightarrow Z$ and $\phi: Z \rightarrow Y$ coincides with $V_y^{k+1}$ as in Lemma~\ref{V}. Therefore, there exists $l \in \NN$ such that $\widetilde{W^{l-1}_y} \rightarrow Y$ is surjective. Hence ${W^{l}_y} \rightarrow X$ is also surjective.
Then, so is $A_2({W^{l}_y})_{\QQ} \rightarrow A_2(X)_{\QQ}$. Thus, to prove $A_2(X)_{\QQ} \cong \QQ$, we only have to show that the rank of  $A_2({W^{l}_y})_{\QQ} \rightarrow A_2(X)_{\QQ}$ is at most $1$.  
Since ${W^{k}_y}$ and $\widetilde{W^{k}_y}$ are rationally connected, $A_0({W^{k}_y})$ and $A_0(\widetilde{W^{k}_y})$ are isomorphic to the ring of integers $\ZZ$. Then it follows from Lemma~\ref{p1} below that $A_1({W^{k}_y})$ and $A_1(\widetilde{W^{k}_y})$ are generated by curves whose images in $Z$ are either a fiber of $\pi$ or a fiber of $\phi$. Furthermore, $A_2({W^{k}_y})$ and $A_2(\widetilde{W^{k}_y})$ are generated by surfaces whose images in $Z$ are either a curve or surfaces of the form $\phi^{-1}(\phi(C))$ for some fiber $C$ of $\pi$ or $\pi^{-1}(\pi(C'))$ for some fiber $C'$ of $\phi$. Hence the rank of  $A_2({W^{l}_y})_{\QQ} \rightarrow A_2(X)_{\QQ}$ is at most $1$. Since $A_2(X)_{\QQ} \rightarrow N_2(X)_{\QQ}$ is surjective, $N_2(X)_{\QQ}$ is also isomorphic to $\QQ$. Thus we obtain $N^2(X)_{\QQ} \cong \QQ$.
\end{proof}

\begin{lem}\label{p1} Let $p: W' \rightarrow W$ be a $\PP^1$-bundle with a section $\sigma : W \rightarrow W'$. Then any $\gamma \in A_k(W')$ is of the form $\gamma = \sigma_{\ast}\alpha + p^{\ast}\beta$ for some $\alpha \in A_k(W)$ and $\beta \in A_{k-1}(W)$.
\end{lem}

\begin{proof} See \cite[Theorem~3.3]{ful}.
\end{proof}

\section{Computation of the discriminant $\Delta(\E)$}

Throughout this section, we work under the following assumptions:

\begin{Ass}\label{A} \rm Let $X$ and $Y$ be $n$-dimensional Fano manifolds of $\rho=1$ and $\E$ a normalized rank $2$ vector bundle over $X$, i.e., $c_1:=c_1(\E)=0$ or $-1$ (when the Picard group of $X$ is identified with $\ZZ$). Assume that $\pi: Z=\PP(\E) \rightarrow X$ admits another smooth morphism $\phi : Z \rightarrow Y$ whose fibers are isomorphic to $\PP^1$ and $n:=\dim X \geq 2$.

\end{Ass}

\begin{nota}\label{nota} \rm
\begin{itemize}
\item $H_X$ (resp. $H_Y$): the ample generator of Pic$(X)$ (resp. Pic$(Y)$). Note that $X$ and $Y$ are Fano manifolds of $\rho=1$ and hence ${\rm Pic}(X) \cong \ZZ$ and ${\rm Pic}(Y) \cong \ZZ$.
\item $i_X$ (resp. $i_Y$): the Fano index of $X$ (resp. $Y$).
\item $H:=\pi^{\ast}H_X, H':=\phi^{\ast}H_Y$.
\item $d_X:=H_X^n, d_Y:=H_Y^n$.
\item $f$ (resp.  $f'$): a fiber of $\pi$ (resp. $\phi$).
\item $\mu:=H.f', {\mu}':=H'.f$.
\item $\Delta(\E):=c_1^2(\E)-4c_2(\E)$
\item $\Sigma$: an effective cycle on $X$ of codimension $2$ such that $N^2(X)_{\QQ} =\QQ\Sigma$ (cf. Proposition~\ref{A2})
\item $c_2(\E)=:c_2\Sigma, H_X^2=:d\Sigma, \Delta(\E)=:(d\Delta)\Sigma$. 
\item $K_{\pi}:=K_Z-\pi^{\ast}K_X$.
\item $L$: a divisor associated with the tautological line bundle of $\PP(\E)$.
\item $\tau:=\tau(\E)$: the unique real number such that $-K_{\pi}+\tau H$ is nef but not ample. 
\item $\upsilon:=\upsilon(\E)$: the unique real number such that $-K_{\pi}+\upsilon H$ is pseudoeffective but not big.
\end{itemize}
\end{nota}

\begin{rem}\label{rem} The following holds:
\begin{enumerate}
\item $\tau \geq \upsilon$.
\item $K_{\pi}^2=\pi^{\ast}\Delta(\E)=\Delta H^2$.
\item $\E$ is not trivial.
\end{enumerate}
\end{rem}

\begin{proof} $\rm (i)$ If $\tau < \upsilon$, then $-K_{\pi}+\upsilon H$ is ample. This contradicts the definition of  $\upsilon$. \\
$\rm (ii)$ By using the Chern-Wu relation  
\begin{eqnarray}
L^2-\pi^{\ast}c_1(\E).L+\pi^{\ast}c_2(\E)=0,
\end{eqnarray}
a direct computation implies that $K_{\pi}^2=\pi^{\ast}\Delta(\E)=\Delta H^2$. \\
$\rm (iii)$ Assume that $\E$ is trivial. Then $Z=X \times \PP^1$, in particular, $Z$ is a Fano manifold. So $\phi$ is a $K_Z$-negative extremal contraction, hence $Y=\PP^1$. However it contradicts $\dim Y=n \geq 2$. 
\end{proof}

We review the definition of (semi)stability of vector bundles and some results in \cite{MOS}.
\begin{defi} \rm Under the same setting as in Assumptions~\ref{A}, let $A$ be an ample divisor on $X$. Then $\E$ is said to be {\it stable} (resp. {\it semistable}) if, for any line bundle $\L \subset \E$, 
\begin{eqnarray}
c_1(\L).A^{n-1} < \frac{1}{2}c_1(\E).A^{n-1}~ ({\rm resp}.~ c_1(L).A^{n-1} \le \frac{1}{2}c_1(\E).A^{n-1}). \nonumber
\end{eqnarray}
\end{defi}

\begin{them}\label{B} Let $(X,\E)$ be as in Assumptions~\ref{A}. If $\E$ is semistable, then, for an ample divisor $A \in {\rm Pic}(X)$, we have   
\begin{eqnarray}
\Delta(\E).A^{n-2} \leq 0. \nonumber
\end{eqnarray}
\end{them}

\begin{proof} This follows from the Bogomolov inequality and the Mehta-Ramanathan theorem.
\end{proof}

\begin{them}[{\cite[Theorem~2.3, Proposition~3.5, Remark~3.6]{MOS}}]\label{MOS} Let $(X,\E)$ be as in Assumptions~\ref{A}. Then the following holds:
\begin{enumerate}
\item $\tau \geq 0$, and the equality holds if and only if $\E \cong \mathscr{O}_X^{\oplus2}$.
\item If $\E$ is not semistable, then $\upsilon \leq 0$, and the equality holds if and only if $\E$ is strictly semistable.
\end{enumerate}
\end{them}

Thanks to Proposition~\ref{A2}, the same argument as in \cite[Proposition~4.12]{MOS} can be applied to our case. In particular, we obtain the following Proposition~\ref{tau} and \ref{key}. For the readers convenience, we recall their argument. 

\begin{pro}[{cf. \cite[Proposition~4.12]{MOS}}]\label{tau} Under the setting as in Assumptions~\ref{A}, the following holds:
\begin{enumerate}
\item $\tau=\upsilon=i_X-\frac{2}{\mu} \in \QQ_{>0}$, and  
\item $\E$ is stable.
\end{enumerate}
\end{pro}

\begin{proof} 
$\rm(i)$ Since $\rho_Z=2$, the Kleiman-Mori cone of $Z$ is spanned by $[f]$ and $[f']$. Furthermore, we have $-K_Z.f=-K_Z.f'=2$. By Kleiman's criterion for ampleness, this implies that $-K_Z$ is ample, that is, $Z$ is a Fano manifold. So the nef cone of $Z$ is a rational polyhedral cone. This implies that $\tau$ is a rational number. It follows from Kawamata-Shokurov base point free theorem that $-K_{\pi}+C H$ is semiample. Then it turns out that $\phi$ is defined by the linear system $|m(-K_{\pi}+\tau H)|$ if $m$ is sufficiently large and divisible. This implies that $(-K_{\pi}+\tau H).f'=0$. Thus we see that $\tau=i_X-\frac{2}{\mu}$. Furthermore, since $\phi$ is a morphism of relative dimension $1$, $-K_{\pi}+\tau H$ is nef but not big. This means that $\tau \leq \upsilon$. By combining Remark~\ref{rem} $\rm (i)$, we get $\tau=\upsilon$. If $\tau=0$, then $\E$ is trivial by Theorem~\ref{MOS} $\rm (i)$. However it contradicts Remark~\ref{rem} $\rm (iii)$.  \\ 
$\rm (ii)$ Since we have $\tau>0$,  Theorem~\ref{MOS} $\rm (iii)$ concludes that $\E$ is stable. 

\end{proof}

\begin{pro}[{cf. \cite[Proposition~4.4]{MOS}}]\label{key} Under the setting as in Assumptions~\ref{A}, the following holds:
\begin{enumerate}
\item $\Delta<0$,
\item $\sqrt{-\Delta}=\tau {\rm tan}(\frac{\pi}{n+1})$, and 
\item $n=2,3$ or $5$.
\end{enumerate}
\end{pro}

\begin{proof} $\rm (i)$ By Proposition~\ref{tau}, $\E$ is stable. Then $\Delta \leq 0$ by Theorem~\ref{B}. Again by Proposition~\ref{tau},   $-K_{\pi}+\tau H$ is nef but not big. So we have $(-K_{\pi}+\tau H)^{n+1}=0$. Since $K_{\pi}^2=\Delta H^2$ (see Remark~\ref{rem}), $H^{n+1}=0$ and $-K_{\pi}.H^n>0$,  $(-K_{\pi}+\tau H)^{n+1}=0$ is equivalent to 
\begin{eqnarray}\label{1}
\sum_{\stackrel{i=0}{i\equiv1 (2)}}^{n+1}\binom{n+1}{i}\tau^{n+1-i}\Delta^{\frac{i-1}{2}}=0. 
\end{eqnarray}
If $\Delta=0$, then $\tau^n=0$ by (\ref{1}). It means that $\tau=0$. However this contradicts Proposition~\ref{tau} $\rm (i)$. As a consequence, we have $\Delta<0$.\\
$\rm (ii)$ From the above equality~(\ref{1}), we obtain  
\begin{eqnarray}\label{2}
(\tau+\sqrt{\Delta})^{n+1}-(\tau-\sqrt{\Delta})^{n+1}=0.
\end{eqnarray}
We denote the argument of the complex number $(\tau + \sqrt{\Delta})^{n+1}$ by ${\rm arg}\left(\tau + \sqrt{\Delta}\right)^{n+1} \in [0, 2\pi)$. Then (\ref{2}) is equivalent to 
\begin{eqnarray}\label{3}
{\rm arg}\left(\tau + \sqrt{\Delta}\right)=0~{\rm or~} \frac{\pi}{n+1}.
\end{eqnarray}
Since we have $\Delta<0$ by $\rm (i)$, (\ref{3}) implies 
\begin{eqnarray}
\sqrt{-\Delta}=\tau {\rm tan}\left(\frac{\pi}{n+1}\right).\nonumber
\end{eqnarray}
\\
$\rm (iii)$ From $\rm (ii)$, we obtain 
\begin{eqnarray}
 {\rm tan}^2\left(\frac{\pi}{n+1}\right)=\frac{{-\Delta}}{\tau^2} \in \QQ. \nonumber
\end{eqnarray}
The algebraic degree of ${\rm tan}\left(\frac{\pi}{n+1}\right)$ over $\QQ$ is known (see \cite[pp. 33-41]{Ni} and \cite[Proposition~2]{Ca}). Then we see that $n=2, 3$ or $5$.

\end{proof}



On the other hand, we give another description of $\Delta(\E)$ via a computation of the total Chern class $c(\pi^{\ast}\E)$. First, we prepare the following lemma. 

\begin{lem}\label{ab} Under the setting as in Assumptions~\ref{A}, 
let $\sigma$ denote the restriction of $\pi$ to $f' \cong \PP^1$. If $\sigma^{\ast}\E \cong \mathscr{O}_{\PP^1}(a)\oplus\O_{\PP^1}(b)$ ($a \geq b$), then we have    
\begin{eqnarray}
(a,b)=\left(-1+\frac{(c_1+i_X)\mu}{2}, 1+{\frac{(c_1-i_X)\mu}{2}}\right).  \nonumber
\end{eqnarray}
\end{lem}

\begin{proof} Let us consider a $\PP^1$-bundle $\sigma^{\ast}Z \cong \PP(\mathscr{O}_{\PP^1}(a)\oplus\O_{\PP^1}(b))$ over $f' \cong \PP^1$. Then $\sigma^{\ast}f'$ is an exceptional curve on $\sigma^{\ast}Z$. It implies that $\sigma^{\ast}f' \cong \PP(\mathscr{O}_{\PP^1}(b))$. Hence $b=L.f'=1+{\frac{(c_1-i_X)\mu}{2}}$. On the other hand, we have $a+b=c_1\mu$. Thus $a=-1+\frac{(c_1+i_X)\mu}{2}$. 
\end{proof}

\begin{lem}\label{c} Under the setting as in Assumptions~\ref{A},  
the total Chern class $c(\pi^{\ast}\E)$ is given by
\begin{eqnarray}\label{ch}
c(\pi^{\ast}\E)=1+\frac{1}{\mu}(a+b)H+\left(\frac{ab}{\mu^2}H^2+\frac{a-b}{\mu\mu'}HH'-\frac{1}{{\mu'}^2}{H'}^2\right). \nonumber
\end{eqnarray}
\end{lem}

\begin{proof} Let $P$ be the kernel of $\pi^{\ast}\E \rightarrow L$. Then we have an exact sequence 
\begin{eqnarray}
0 \rightarrow P \rightarrow \pi^{\ast}\E \rightarrow L \rightarrow 0. 
\end{eqnarray} 
In general, any saturated subsheaf of a locally-free sheaf is again locally-free. So $P$ is a line bundle. 
Since $L|_f \cong \mathscr{O}_{\PP^1}(1)$ and $L|_{f'} \cong \mathscr{O}_{\PP^1}(b)$, we see that $P|_f = \ker(\pi^{\ast}\E|_f \rightarrow \mathscr{O}_{\PP^1}(1)) \cong  \mathscr{O}_{\PP^1}(-1)$ and $P|_{f'} = \ker(\pi^{\ast}\E|_{f'} \rightarrow \mathscr{O}_{\PP^1}(b)) \cong  \mathscr{O}_{\PP^1}(a)$. Remark that $A^1(Z)_\QQ = \langle H, H'\rangle_\QQ$. Hence we obtain 
\begin{eqnarray}
c(L) &=& 1+\frac{b}{\mu}H+\frac{1}{\mu'}H', \nonumber\\
c(P) &=& 1+\frac{a}{\mu}H-\frac{1}{\mu'}H', \nonumber
\end{eqnarray} 
and 
\begin{eqnarray}
c(\pi^{\ast}\E) &=& c(L)\cdot c(P) =\left(1+\frac{b}{\mu}H+\frac{1}{\mu'}H'\right)\cdot \left(1+\frac{a}{\mu}H-\frac{1}{\mu'}H'\right) \nonumber\\ 
&=& 1+\frac{1}{\mu}(a+b)H+\left(\frac{ab}{\mu^2}H^2+\frac{a-b}{\mu\mu'}HH'-\frac{1}{{\mu'}^2}{H'}^2\right).\nonumber
\end{eqnarray} 
\end{proof}

By using Lemmas~\ref{c} and \ref{eq} below, the equality (\ref{ch}) will be rewritten in more simple form in Proposition~\ref{C}.

\begin{lem}\label{e} Under the setting as in Assumptions~\ref{A}, let $\nu$ be the restriction of $\phi$ to $f \cong \PP^1$ and $\zeta$ a projection $\nu^{\ast}Z \rightarrow Z$. If $N_{\nu^{\ast}f/\nu^{\ast}Z} \cong \mathscr{O}_{\PP^1}(-e)$, then $e>0$ and we have
\begin{eqnarray}
(\zeta^{\ast}H)^2=\frac{\mu e}{\mu'}(\zeta^{\ast}H\zeta^{\ast}H') . \nonumber
\end{eqnarray}
\end{lem}

\begin{proof} We consider a $\PP^1$-bundle $\psi: \nu^{\ast}Z \rightarrow f \cong \PP^1$. Then $\nu^{\ast}f$ is an exceptional curve on $\nu^{\ast}Z$. From $N_{\nu^{\ast}f/\nu^{\ast}Z} \cong \mathscr{O}_{\PP^1}(-e)$, we obtain $e>0$. Furthermore, we see that $\nu^{\ast}Z \cong \PP(\mathscr{O}_{\PP^1} \oplus \mathscr{O}_{\PP^1}(-e))$ and $\nu^{\ast}f \cong \PP(\mathscr{O}_{\PP^1}(-e))$. Let $M$ be the tautological line bundle of $\nu^{\ast}Z \cong \PP(\mathscr{O}_{\PP^1} \oplus \mathscr{O}_{\PP^1}(-e))$ and $Q$ the kernel of $\psi^{\ast}\left(\mathscr{O}_{\PP^1} \oplus \mathscr{O}_{\PP^1}(-e)\right) \rightarrow M$. Then we have an exact sequence 
\begin{eqnarray}
0 \rightarrow Q \rightarrow \psi^{\ast}\left(\mathscr{O}_{\PP^1} \oplus \mathscr{O}_{\PP^1}(-e)\right) \rightarrow M \rightarrow 0. \nonumber
\end{eqnarray} 

Then we see that $M|_{\nu^{\ast}f} \cong \mathscr{O}_{\PP^1}(-e)$ and $M|_{\nu^{\ast}f'} \cong \mathscr{O}_{\PP^1}(1)$. These imply that 
\begin{eqnarray}
Q|_{\nu^{\ast}f} &=& \ker\left(\psi^{\ast}\left(\mathscr{O}_{\PP^1} \oplus \mathscr{O}_{\PP^1}(-e)\right)|_{\nu^{\ast}f}  \rightarrow \mathscr{O}_{\PP^1}(-e)\right) \cong \mathscr{O}_{\PP^1}, ~{\rm and} \nonumber \\ 
Q|_{\nu^{\ast}f'} &=& \ker\left(\psi^{\ast}\left(\mathscr{O}_{\PP^1} \oplus \mathscr{O}_{\PP^1}(-e)\right)|_{\nu^{\ast}f'}  \rightarrow \mathscr{O}_{\PP^1}(1)\right) \cong \mathscr{O}_{\PP^1}(-1).  \nonumber
\end{eqnarray}
Remark that $A^1(\nu^{\ast}Z)_\QQ = \langle \nu^{\ast}H, \nu^{\ast}H'\rangle_\QQ$. Hence we obtain 
\begin{eqnarray}
c(M)&=&1+\frac{1}{\mu}\zeta^{\ast}H-\frac{e}{\mu'}\zeta^{\ast}H',\nonumber \\
c(Q)&=&1-\frac{1}{\mu}\zeta^{\ast}H \nonumber.
\end{eqnarray} 
and 
\begin{eqnarray} 
c(\psi^{\ast}\left(\mathscr{O}_{\PP^1} \oplus \mathscr{O}_{\PP^1}(-e)\right))&=&c(M)\cdot c(Q) =\left(1+\frac{1}{\mu}\zeta^{\ast}H-\frac{e}{\mu'}\zeta^{\ast}H' \right)\cdot\left(1-\frac{1}{\mu}\zeta^{\ast}H\right) \nonumber \\ 
&=& 1-\frac{e}{\mu'}\zeta^{\ast}H'+\left(-\frac{1}{\mu^2}(\zeta^{\ast}H)^2+\frac{e}{\mu\mu'}\zeta^{\ast}H\zeta^{\ast}H' \right) \nonumber.
\end{eqnarray} 
Furthermore, we obtain
\begin{eqnarray}
-\frac{1}{\mu^2}(\zeta^{\ast}H)^2+\frac{e}{\mu\mu'}\zeta^{\ast}H\zeta^{\ast}H' =c_2\left(\psi^{\ast}\left(\mathscr{O}_{\PP^1} \oplus \mathscr{O}_{\PP^1}(-e)\right)\right)=\psi^{\ast}\left(c_2\left(\mathscr{O}_{\PP^1} \oplus \mathscr{O}_{\PP^1}(-e)\right)\right)=0 \nonumber.
\end{eqnarray} 
As a consequence, we get
\begin{eqnarray}
(\zeta^{\ast}H)^2=\frac{\mu e}{\mu'}(\zeta^{\ast}H\zeta^{\ast}H') \nonumber
\end{eqnarray}
as desired.
\end{proof}

\begin{lem}\label{eq}  Under the setting as in Assumptions~\ref{A}, we have
\begin{eqnarray}
\frac{a-b}{e\mu^2}H^2 = \frac{a-b}{\mu\mu'}HH'-\frac{1}{\mu'^2}{H'}^2 \in N^2(Z)_{\QQ}. \nonumber
\end{eqnarray}
\end{lem}

\begin{proof} In Lemma~\ref{c}, we have seen that $\pi^{\ast}\left(c_2\left(\E \right)\right)=c_2(\pi^{\ast}\E)=\frac{ab}{\mu^2}H^2+\frac{a-b}{\mu\mu'}HH'-\frac{1}{{\mu'}^2}{H'}^2$. Since we have $N^2(X)_{\QQ} \cong \QQ$, there exists $g \in \QQ$ such that
\begin{eqnarray}\label{g} 
gH^2+\frac{a-b}{\mu\mu'}HH'-\frac{1}{{\mu'}^2}{H'}^2=0 \in N^2(Z)_{\QQ}.
\end{eqnarray}
Pulling back to $\nu^{\ast}Z$ by $\zeta$, we obtain
\begin{eqnarray} 
\zeta^{\ast}(gH^2+\frac{a-b}{\mu\mu'}HH')=\zeta^{\ast}\left(gH^2+\frac{a-b}{\mu\mu'}HH'-\frac{1}{{\mu'}^2}{H'}^2\right)=0 \in N^2(\nu^{\ast}Z)_{\QQ}. \nonumber
\end{eqnarray}
By Lemma~\ref{e}, $(\zeta^{\ast}H)^2=\frac{\mu e}{\mu'}(\zeta^{\ast}H\zeta^{\ast}H')$. It turns out that 
\begin{eqnarray} 
\left(\frac{g\mu e}{\mu'}+\frac{a-b}{\mu\mu'}\right)\zeta^{\ast}H\zeta^{\ast}H'=0 \in N^2(\nu^{\ast}Z)_{\QQ}. \nonumber
\end{eqnarray}
Hence we have $g=\frac{b-a}{\mu^2e}$. Substituting this in the equation~(\ref{g}), we obtain 
\begin{eqnarray}
\frac{a-b}{e\mu^2}H^2 = \frac{a-b}{\mu\mu'}HH'-\frac{1}{\mu'^2}{H'}^2 \in N^2(Z)_{\QQ}. \nonumber
\end{eqnarray} 
\end{proof}

By combining Lemmas~\ref{c} and \ref{eq}, we get the following:

\begin{pro}\label{C} Under the setting as in Assumptions~\ref{A}, 
the total Chern class $c(\pi^{\ast}\E)$ is given by
\begin{eqnarray}
c(\pi^{\ast}\E)=1+\frac{1}{\mu}(a+b)H+\left(\frac{ab}{\mu^2}+\frac{a-b}{e\mu^2}\right)H^2 \in 1\oplus N^1(Z) \oplus N^2(Z)_{\QQ}. \nonumber
\end{eqnarray}
\end{pro}

\begin{pro}\label{delta}  Under the setting as in Assumptions~\ref{A}, $e$ is defined by $N_{\nu^{\ast}f/\nu^{\ast}Z} \cong \mathscr{O}_{\PP^1}(-e)$ as in Lemma~\ref{e}. Then we have
\begin{eqnarray}
\Delta=\tau^2-\frac{4\tau}{e \mu}. \nonumber
\end{eqnarray}
\end{pro}

\begin{proof} From the definition of $\Delta$ and Proposition~\ref{C}, 
\begin{eqnarray}
\Delta H^2 &=& c_1(\pi^{\ast}\E)^2-4c_2(\pi^{\ast}\E) \nonumber \\ 
&=& \left(\frac{1}{\mu}(a+b)H\right)^2-4\left(\frac{ab}{\mu^2}+\frac{a-b}{e\mu^2}\right)H^2
=\frac{a-b}{\mu^2 e}\left(e\left(a-b\right)-4\right)H^2. \nonumber
\end{eqnarray}
By Lemma~\ref{ab}, $a-b=i_X\mu-2=\tau\mu$. Thus, we obtain
$\Delta =\tau^2-\frac{4\tau}{e \mu}$ as desired.
\end{proof}



\section{Proof of Theorem~\ref{MT}}

In this section, we prove Theorem~\ref{MT}. It is sufficient to work under the same setting as in Assumptions~\ref{A}.
Then it follows from Proposition~\ref{key} that $n=2,3$ or $5$.

\begin{them}\label{n2} With the same setting as in Assumptions~\ref{A}, if $n=2$, then $(X,\E)$ is isomorphic to $(\PP^2, T_{\PP^2})$.    
\end{them}
 
\begin{proof}  Since a Fano surface of $\rho=1$ is isomorphic to $\PP^2$, we have $X \cong \PP^1$ and $i_X=3$. By virtue of Proposition~\ref{key}~(ii) and Proposition~\ref{delta}, we have $\Delta=-3\tau^2$ and $\Delta=\tau^2-\frac{4\tau}{e \mu}$. Recall that $\tau>0$ by Proposition~\ref{tau}~(i). Thus $(3\mu-2)e=\tau\mu e=1$. 
 Hence $(i_X, \mu, e)=(3,1,1)$, $\tau=1$ and $\Delta=-3$. By Lemma~\ref{ab}, $(c_1+i_X)\mu$ is divisible by $2$. This implies that $c_1=-1$. Here we take a point on $X \cong \PP^2$ as a base $\Sigma$ of $N^2(X)_{\QQ}$. Then $d=1$. Since $\Delta=-3$ and $c_1=-1$, we see that $c_2=1$. Hence $\E$ is a rank $2$ stable vector bundle over $\PP^2$ with $(c_1, c_2)=(-1, 1)$. Then $\E$ is isomorphic to $T_{\PP^2}$ by \cite{Hu}.
\end{proof}


\begin{lem}\label{n3} With the same setting as in Assumptions~\ref{A}, if $n=3$, then $(i_X, \mu, e)=(4,1,1), (3,1,2), (2,2,1), (1,3,2)$ or $ (1,4,1)$. 
\end{lem}

\begin{proof} By virtue of Proposition~\ref{key}~(ii) and Proposition~\ref{delta}, we have $\Delta=-\tau^2$ and $\Delta=\tau^2-\frac{4\tau}{e \mu}$. Recall that $\tau>0$ by Proposition~\ref{tau}~(i). Thus $(i_X\mu-2)e=\tau\mu e=2$. This implies that $(i_X, \mu, e)=(4,1,1), (3,1,2), (2,2,1), (1,3,2)$ or $ (1,4,1)$.
\end{proof}

\begin{lem}\label{221}Under the same setting as in Lemma~\ref{n3}, if $(i_X, \mu, e)=(2,2,1)$, then $c_1=0$. 
\end{lem}

\begin{proof} In this case, $X$ is a del Pezzo $3$-fold of $\rho=1$. So $H^4(X, \ZZ)$ is generated by a line $l$ on $X$. Here a line means a rational curve with $H_X.l=1$. We take $l$ as a base $\Sigma$ of $N^2(X)_{\QQ}$. Then $c_2$ is an integer.

Now assume the contrary of our claim, that is, $c_1=-1$. Then Riemann-Roch theorem tells us that $c_2$ is even. On the other hand, we see that $\Delta=-1$. Thus, we obtain $d_X=d=2c_2$. From the classification of del Pezzo $3$-fold of $\rho=1$ \cite{Isk}, it follows that $d_X \leq 5$. It turns out that $(d_X, c_2)=(4,2)$. Again, according to \cite{Isk}, $X$ is a complete intersection of two quadric $4$-folds in $\PP^5$. 

Consider a morphism $\pi \circ \zeta : \nu^{\ast}Z \rightarrow X$. Since $\nu^{\ast}Z \cong \PP(\mathscr{O}_{\PP^1}\oplus \mathscr{O}_{\PP^1}(-1))$, $\pi \circ \zeta$ factors through $g: \PP^2 \rightarrow  X$. This can be obtained by taking the Stein factorization of $\pi \circ \zeta=g \circ h$, where $h: \nu^{\ast}Z \rightarrow \PP^2$ is a blow-up at a point $o \in \PP^2$. Remark that $h$ sends every fiber of $\nu^{\ast}Z \rightarrow f \cong \PP^1$ to a line through $o \in \PP^2$. Hence, for a line $l_o$ through $o \in \PP^2$,  we have $g^{\ast}H_X.l_o=\mu=2$. This implies that $g^{\ast}\mathscr{O}_X(H_X) \cong \mathscr{O}_{\PP^2}(2)$. Let $S$ denote the image of $g$. Then the degree of $S \subset \PP^5$ satisfies that $\deg (g) \deg (S)=4$.
On the other hand, $S$ is a member of the linear system $|\mathscr{O}_X(sH_X)|$ for some $s>0$. So we have $\deg(S)=4s$. Hence we see that $(s, \deg(g))=(1,1)$. Then, it is easy to see that $S$ is smooth, that is, $S \cong \PP^2$. However this contradicts the adjunction formula. 

\end{proof}

By the same way as in Lemma~\ref{n3}, we can prove the following:

\begin{lem}\label{n5} With the same setting as in Assumptions~\ref{A}, if $n=5$, then $(i_X, \mu, e)=(5,1,1), (3,1,3), (1,5,1),$ or $ (1,3,3)$. 
\end{lem}

Now we prove the remaining part of Theorem~\ref{MT}~$\rm (I)$.

\begin{pro}\label{n35} Under the same setting as in Assumptions~\ref{A}, there exists a rank $2$ vector bundle $\E'$ on $Y$ such that $Z=\PP_Y(\E')$.
\end{pro}

\begin{proof} If $n=2$, this follows from Theorem~\ref{n2}. Thus, we deal with the cases where $n=3$ and $5$. From the above lemmas, it follows that \begin{eqnarray}(n, i_X, \mu, e)&=&(3,4,1,1), (3,3,1,2), (3,2,2,1), (3,1,3,2), (3,1,4,1), (5,5,1,1),\nonumber \\ && (5,3,1,3), (5,1,5,1)~{\rm or}~(5,1,3,3). \nonumber
\end{eqnarray}
It is enough to find a line bundle $V$ on $Z$ which satisfies $V.f'=1$. Indeed, $\E':=\phi_{\ast}\mathscr{O}_Z(V)$ satisfies the property desired.
Hence it is sufficient to deal with the case where $\mu:=H.f' \neq 1$, that is, $(n, i_X, \mu, e)=(3,2,2,1), (3,1,3,2), (3,1,4,1), (5,1,5,1),$ and $ (5,1,3,3)$. 

Recall that $L.f'=1+{\frac{(c_1-i_X)\mu}{2}}$ due to Lemma~\ref{ab}. If $(n, i_X, \mu, e)=(3,2,2,1)$, then, by Lemma~\ref{221}, we have $c_1=0$. This means $L.f'=-1$. Hence $V:=H \otimes L$ satisfies $V.f'=1$. 
If $(n, i_X, \mu, e)=(3,1,3,2)$, $(5,1,5,1)$ or $(5,1,3,3)$, then we see that $L.f'=-2, -4$, $-2$, respectively. It turns out that $V:=H \otimes L$ satisfies $V.f'=1$. If $(n, i_X, \mu, e)=(3,1,4,1)$, then $L.f'=-1$ or $-3$. Hence we can take $H \otimes L^{\otimes 3}$ or $H \otimes L$ as $V$.
\end{proof}

According to this proposition, we see that $Z$ admits double $\PP^1$-bundle structures $\pi: Z \rightarrow X$ and $\phi: Z \rightarrow Y$. By symmetry of $X$ and $Y$, all the results on $X$ as above also hold for $Y$. By the same way as in Notation~\ref{nota}, we define rational numbers $c_1', c_2', d'$ and $ \Delta'$ for $Y$ and $\E'$. Here $\E'$ may be normalized. Moreover, $K_{\phi}$ and $L'$ stand for the relative canonical divisor and a divisor associated with the tautological line bundle of $\PP(\E')$, respectively. Then we define $\tau':=\tau(\E')$ and $\upsilon':=\upsilon(\E')$ as in Notation~\ref{nota}. 
Applying the argument as in \cite{MOS}, we complete the proof of Theorem~\ref{MT} as follows:

\begin{them}\label{conc} Under the same setting as in Assumptions~\ref{A} and the above, 
if $n \geq 3$, then $((X, \E), (Y, \E'))$ is isomorphic to
$((\PP^3, \N),(Q^3, \S))$ or $((Q^5, \C), (K(G_2), \Q))$ up to changing the pairs $(X, \E)$ and $(Y, \E')$. 
\end{them}

\begin{proof} As we have seen in Proposition~\ref{tau}, $\tau=i_X-\frac{2}{\mu}$ and $\tau'=i_Y-\frac{2}{\mu'}$. Then we obtain the following table:
\begin{center}
\begin{tabular}{|c||c|c|c|c|}
\hline
 $ $ & $H$ & $H'$ & $L$ & $L'$ \\ \hline \hline
 $f$ & $0$ & $\mu'$ & $1$ &  $1+\frac{(c_1'-i_Y){\mu}'}{2}$   \\ \hline
 $f'$ & $\mu$ & $0$ & $1+\frac{(c_1-i_X){\mu}}{2}$ &  $1$ \\
   \hline
\end{tabular}

\end{center}  
This table represents intersection numbers of divisors in the first row and $f$ or $f'$. For example, $H.f=0$ and $H.f'=\mu$ etc. Applying this table, we obtain
\begin{eqnarray}
\left\{
\begin{array}{l}
H'=-\frac{\mu'}{2}(c_1-\tau)H+\mu'L \\
L'=\{-\frac{\mu'}{4}(c_1-\tau)(c_1'-\tau')+\frac{1}{\mu}\}H+\frac{\mu'}{2}(c_1'-\tau')L.
\end{array}
\right.
\end{eqnarray}

Since $\{H,L\}$ and $\{H',L'\}$ are $\ZZ$-bases of ${\rm Pic}(\PP(\E))$, the determinant of the matrix of base change is equal to $1$ or $-1$. This implies that $\mu=\mu'$. Hence we can write $H'=\frac{\mu}{2}(-K_{\pi}+\tau H)$.  Furthermore, we get
\begin{eqnarray}\label{8}
\frac{d_Y}{d_X}=(\frac{\mu}{2})^n\frac{(-K_\pi+\tau H)^nH/\mu}{-K_\pi H^n/2}=(\frac{\mu}{2})^{n-1}\frac{{\rm im}((\tau+\sqrt{\Delta})^n)}{\sqrt{-\Delta}}. 
\end{eqnarray} 
Since we have $\sqrt{-\Delta}=\tau {\rm tan}(\frac{\pi}{n+1})$, (\ref{8}) is equivalent to 
\begin{eqnarray}\label{}
\frac{d_Y}{d_X}=(\frac{\tau\mu}{2{\rm cos}(\pi/n+1)})^{n-1}. \nonumber
\end{eqnarray} 
By symmetry of $X$ and $Y$, we get a similar equation 
\begin{eqnarray}\label{}
\frac{d_X}{d_Y}=(\frac{\tau'\mu'}{2{\rm cos}(\pi/n+1)})^{n-1}. \nonumber
\end{eqnarray}
These equations imply 
\begin{eqnarray}\label{10}
(\frac{\tau\mu\tau'\mu'}{4{\rm cos}^2(\pi/n+1)})^{n-1}=1.
\end{eqnarray} 
Since $\tau, \tau'>0$ and $\mu=\mu'>0$, $(\ref{10})$ provides
\begin{eqnarray}\label{11}
(i_X\mu-2)(i_Y\mu-2)=\tau\tau'\mu^2=\left\{ \begin{array}{ll}
2 & (n=3) \\
3 & (n=5) \\
\end{array} \right.
\end{eqnarray}

We may assume that $i_X \geq i_Y$. From (\ref{11}), we have $(i_X,i_Y,\mu)=(4,3,1)$ provided $n=3$. Hence we see that $X \cong \PP^3$. Here we take a line on $X \cong \PP^3$ as a base $\Sigma$ of $N^2(X)_{\QQ}$. Then $d=1$. On the other hand, if $n=5$, then we have $(i_X,i_Y,\mu)=(5,3,1)$. Hence we obtain that $X \cong Q^5$. Here $H^4(X,Q^5) \cong \ZZ$ and we take its positive generator as a base $\Sigma$ of $N^2(X)_{\QQ}$. Then $d=1$. In both cases, easy calculations imply the following table:
\begin{center}
\begin{tabular}{|c||c|c|c|c|c|c|c|}
\hline
 $ n$ & $i_X$ & $d$ & $\mu$ & $\tau$ & $\Delta$ & $c_1$ & $c_2$ \\ \hline \hline
 $3$ & $4$ & $1$ & $1$ &  $2$ & $-4$ & $0$ & $1$   \\ \hline  
 $5$ & $5$ & $1$ & $1$ &  $3$ & $-3$ & $-1$ & $1$   \\ \hline  
\end{tabular}
\end{center}  

Since vector bundles $\N$ and $\C$ are determined by their Chern classes among stable bundles (see \cite[Lemma~4.3.2]{OSS} and \cite{Ota}), hence $(X,\E)$ is isomorphic to $(\PP^3, \N)$ or $(Q^5, \C)$. Then the structure of $(Y, \E')$ is well-known (for instance, see \cite[Proposition~2.6]{SW}, \cite[Example~6.4]{MOS} and \cite[1.3]{Ota}). Consequently, Theorem~\ref{MT} holds. 
\end{proof}

\section{Proof of Theorem~\ref{HM}}

Let $M$ be an $n$-dimensional Fano manifold of $\rho=1$ with nef tangent bundle. Assume that $M$ carries a rational curve $C$ such that $-K_M.C=3$. This assumption implies that $n \geq 2$. Let $\K$ be a minimal rational component of $M$ (see \cite[1.2]{Mok}) and $\pi: \U \rightarrow \K$ its universal family. Denote the evaluation map by $\iota: \U \rightarrow M$.  
\begin{lem}[{\cite[Lemma~1.2.1, Lemma~1.2.2, Corollary~1.3.1]{Mok}}] Under the above setting, the following holds:
\begin{enumerate}
\item $\K$ is a projective manifold of dimension $n$, 
\item every fiber of $\iota$ is isomorphic to $\PP^1$, and
\item $\pi: \U \rightarrow \K$ is a $\PP^1$-bundle.
\end{enumerate}
\end{lem}

Since $\iota$ is a smooth morphism whose fibers are $\PP^1$, one can check that the second Betti numbers satisfy $b_2(\U)=b_2(M)+1$. This implies $b_2(\K)=1$. Furthermore, $\K$ is covered by rational curves which are images of fibers of $\iota$, that is, a uniruled manifold. Consequently, $\K$ is a Fano manifold of $\rho_{\K}=1$. Applying Theorem~\ref{MT}, we obtain Theorem~\ref{HM}.


\begin{thebibliography}{7}

\bibitem{AS} E. Arrondo, I. Sols, {\it Classification of smooth congruence of low degree}, J. Reine Angew. Math. 393 (1989), 199-219.
\bibitem{Ca} J. S. Calcut, {\it Rationality and the tangent function}, preprint, available at http://www.oberlin.edu/faculty/jcalcut/tanpap.pdf.
\bibitem{CP} F. Campana, T. Peternell, {\it Projective manifolds whose tangent bundles are numerically effective}, Math. Ann. 289 (1991), 169-187.
\bibitem{ful} W. Fulton, Intersection Theory, Ergeb. Math. Grenzgeb. (3), Springer-Verlag, Berlin, Heidelberg, New York, 1984.
\bibitem{Ha} R. Hartshorne, Algebraic geometry. Graduate Texts in Mathematics, No. 52. Springer-Verlag, New York-Heidelberg, 1977.
\bibitem{Hu} K. Hulek, {\it Stable rank-2 vector bundles on $\PP^2$ with $c_1$ odd}, Math. Ann. 242, 241-266, (1979).
\bibitem{Hwang} J.M. Hwang, {\it Rigidity of rational homogeneous spaces}, Proceedings of ICM. 2006 Madrid, volume II, European Mathematical Society, 2006, 613-626.
\bibitem{Isk} V. A. Iskovskih, {\it Fano 3-folds.} I, II, Math. USSR Izv. 11 (1977) 485-529; 12 (1978). 496-506.
\bibitem{KMM1} J. Koll\'ar, Y. Miyaoka and S. Mori, {\it Rational curves on Fano varieties}, Proc. Alg. Geom. Conf. Trento, Springer Lecture Notes 1515 (1992) 100-105.
\bibitem{Mok} N. Mok, {\it On Fano manifolds with nef tangent bundles admitting 1-dimensional varieties of minimal rational tangents}, Trans. Amer. Math. Soc. 354 (2002), 2639 -2658.
\bibitem{MOS} R. Mu$\rm \tilde{n}$oz, G. Occhetta, L. Sol$\rm \acute{a}$ Conde, {\it On rank 2 vector bundles on Fano manifolds}, arXiv:1104.1490.
\bibitem{Ni} I. Niven, {\it  Irrational Numbers}, The Carus Mathematical Monographs, no. 1 1, MAA, 1956
\bibitem{OSS} C. Okonek, M. Schneider and H. Spindler, Vector bundles over complex projective space, Progress in Math., vol. 3, Birkh$\rm \ddot{a}$user, Boston, Basel, Stuttgart, 1980.
\bibitem{Ota} G. Ottaviani, {\it On Cayley bundles on the five-dimensional quadric}, Boll. Un. Mat. Ital. A (7) 4 (1990).
\bibitem{SW}  M. Szurek, J. A. Wi\'sniewski, {\it Fano bundles over $\PP^3$ and $Q^3$}, Pacific J. Math. 141 (1990), no. 1, 197-208.
\end{thebibliography}
\end{document}